\documentclass[11pt]{article}

\usepackage[T1]{fontenc}
\usepackage[utf8]{inputenc}
\usepackage{amsmath,amssymb,amsthm}
\usepackage{geometry}
\usepackage{microtype}
\usepackage{xcolor}
\usepackage{hyperref}

\hypersetup{
  colorlinks=true,
  linkcolor=blue!45!black,
  citecolor=blue!45!black,
  urlcolor=blue!45!black,
  pdftitle={Optimal Finite Interval Discrepancy via Binary Refinement},
  pdfauthor={Arthur F. Ramos, David B. Hulak, Ruy J. G. B. de Queiroz}
}

\newtheorem{theorem}{Theorem}[section]
\newtheorem{lemma}[theorem]{Lemma}
\newtheorem{corollary}[theorem]{Corollary}
\newtheorem{proposition}[theorem]{Proposition}
\theoremstyle{remark}
\newtheorem{remark}[theorem]{Remark}

\DeclareMathOperator{\disc}{disc}
\newcommand{\ceil}[1]{\left\lceil #1 \right\rceil}

\title{Optimal Finite Interval Discrepancy via Binary Refinement}
\author{%
Arthur F. Ramos\thanks{\raggedright Microsoft; \texttt{arfreita@microsoft.com}}
\and
David B. Hulak\thanks{\raggedright Independent Researcher; \texttt{dbhulak@gmail.com}}
\and
Ruy J. G. B. de Queiroz\thanks{\raggedright Centro de Inform\'atica, Universidade Federal de Pernambuco; \texttt{ruy@cin.ufpe.br}.}}
\date{}

\begin{document}

\maketitle

\begin{abstract}
DeLeo, Henderschedt, and Wells introduced a finite-horizon version of the classical de Bruijn--Erdos interval discrepancy problem. Starting from the unit interval, one repeatedly splits an existing interval into two until $n$ intervals are present, and one minimizes the largest ratio between the longest and shortest intervals over all intermediate partitions. They constructed the lex-merge strategy, whose discrepancy is $2^{1-1/\ceil{n/2}}$, and conjectured that this value is optimal for every $n$. We prove the conjecture. More generally, we establish a sharp lower bound for arbitrary binary refinement processes of positive masses: any process that starts with one positive mass, repeatedly replaces one mass by two positive masses with the same total, and terminates with $n$ masses must at some stage have largest-to-smallest ratio at least $2^{1-1/\ceil{n/2}}$. The proof tracks the minimum mass under refinement and uses the forced survival of a piece near the midpoint of the process. We also record the corresponding universal lower bound for $r$-ary refinements.
\end{abstract}


\section{Introduction}

A classical problem of de Bruijn and Erd\H{o}s asks how uniformly a sequence of points can successively divide a circle. In the equivalent interval formulation, one begins with a single interval and repeatedly divides one existing interval into two. The quality of a partition is measured by the ratio between its longest and shortest constituent intervals. De Bruijn and Erd\H{o}s proved that, for an infinite splitting process, discrepancy $2$ is asymptotically unavoidable and achievable; see \cite{deBruijnErdos}.

DeLeo, Henderschedt, and Wells \cite{DHW} studied the corresponding finite-horizon optimization problem. If $\mathcal I$ is a partition into positive-length intervals, write
\[
  \disc(\mathcal I)
  = \frac{\max\{|I|: I\in\mathcal I\}}
         {\min\{|I|: I\in\mathcal I\}}.
\]
A strategy of length $n$ is a sequence
\[
  \mathcal S=(\mathcal I_1,\ldots,\mathcal I_n),
\]
where $\mathcal I_i$ consists of $i$ intervals and $\mathcal I_{i+1}$ is obtained from $\mathcal I_i$ by splitting one interval into two. Its discrepancy is
\[
  \disc(\mathcal S)=\max_{1\leq i\leq n}\disc(\mathcal I_i),
  \qquad
  \disc(n)=\min_{\mathcal S}\disc(\mathcal S).
\]

The preprint \cite{DHW} proves the bounds
\[
  2^{1-1/\ceil{n/3}}
  \leq \disc(n)
  \leq 2^{1-1/\ceil{n/2}},
  \tag{1}
\]
where the upper bound is realized by an explicit construction called lex-merge. Its Theorem 3.5 computes the discrepancy of that construction exactly, and its Conjecture 1.4 asserts that the upper bound is optimal for every positive integer $n$.

The purpose of this note is to prove that conjecture. The conceptual core is a
geometry-free lower bound for binary refinements; the exact interval result then
follows by combining it with the lex-merge upper bound. Suppose that a positive
mass is repeatedly replaced by two positive masses whose sum equals the parent
mass. If the spread between the largest and smallest masses is bounded by $D$,
then after each refinement the new minimum is at most $(D/2)$ times the
previous minimum. On the other hand, once the process contains more pieces
than there are refinements remaining, at least one present piece must survive
unchanged to the terminal state. Comparing these two facts gives the optimal
lower bound.

\section{The universal binary-refinement lower bound}
\label{sec:binary}

We formulate the argument independently of intervals. A binary refinement process of length $n$ consists of states
\[
  P_1,P_2,\ldots,P_n,
\]
where $P_i$ contains $i$ distinct pieces, together with positive weights
\[
  w_i(x)>0 \qquad (x\in P_i).
\]
To pass from $P_i$ to $P_{i+1}$, one piece $x\in P_i$ is removed and replaced by two new pieces $y,z$, with
\[
  w_{i+1}(y)+w_{i+1}(z)=w_i(x),
  \tag{2}
\]
while every other piece persists with its weight unchanged. Pieces are regarded as distinct even when they have equal weights, so their identities can be followed through the process. Equivalently, one may assign each piece a finite binary address, appending $0$ and $1$ to the address of a piece when it is split. This provides persistent identities without using interval geometry.

For each state define
\[
  a_i=\min_{x\in P_i}w_i(x),
  \qquad
  b_i=\max_{x\in P_i}w_i(x),
\]
and define the discrepancy of the process by
\[
  D=\max_{1\leq i\leq n}\frac{b_i}{a_i}.
  \tag{3}
\]

The first observation records the effect of one refinement on the minimum weight.

\begin{lemma}
\label{lem:min-decay}
For every $1\leq i<n$,
\[
  a_{i+1}\leq \frac{D}{2}a_i.
  \tag{4}
\]
\end{lemma}

\begin{proof}
Suppose that the piece $x\in P_i$ is replaced by $y,z$. By (3),
\[
  w_i(x)\leq b_i\leq D a_i.
\]
Since the two new weights sum to $w_i(x)$, at least one of them is no larger than half of $w_i(x)$. Therefore
\[
  a_{i+1}
  \leq \min\{w_{i+1}(y),w_{i+1}(z)\}
  \leq \frac{w_i(x)}2
  \leq \frac{D}{2}a_i.
\]
\end{proof}

The second ingredient is purely combinatorial.

\begin{lemma}
\label{lem:survivor}
Let
\[
  m=\ceil{n/2},
  \qquad
  s=n-m+1.
  \tag{5}
\]
At least one piece present in $P_s$ survives unchanged through $P_n$.
\end{lemma}

\begin{proof}
There are $n-s=m-1$ refinements remaining after state $P_s$. On the other hand, $|P_s|=s$. If $n=2q$, then $m=q$ and $s=q+1$; if $n=2q-1$, then $m=q$ and $s=q$. In either case, $s>m-1$.

For every piece of $P_s$ to disappear before the terminal state, each such piece would have to be refined at least once. The first refinement of distinct pieces of $P_s$ must occur at distinct refinement steps, so eliminating all $s$ pieces requires at least $s$ subsequent refinements. Only $m-1<s$ refinements remain. Hence at least one piece of $P_s$ is never refined and persists unchanged through $P_n$.
\end{proof}

The two observations give the sharp obstruction. This geometry-free theorem is
the conceptual core of the note; the interval result will follow by combining
it with the lex-merge construction.

\begin{theorem}[Universal binary-refinement lower bound]
\label{thm:binary}
Every binary refinement process of length $n$ satisfies
\[
  D\geq 2^{1-1/\ceil{n/2}}.
  \tag{6}
\]
\end{theorem}

\begin{proof}
Set $m=\ceil{n/2}$ and $s=n-m+1$. By Lemma~\ref{lem:survivor}, choose a piece $x\in P_s$ that survives unchanged through $P_n$. Since $x\in P_s$ and its weight is unchanged,
\[
  a_s\leq w_s(x)=w_n(x)\leq b_n\leq D a_n.
  \tag{7}
\]
There are exactly $n-s=m-1$ refinements between $P_s$ and $P_n$. Iterating Lemma~\ref{lem:min-decay} yields
\[
  a_n\leq \left(\frac{D}{2}\right)^{m-1}a_s.
  \tag{8}
\]
Combining (7) and (8), and using $a_s>0$, gives
\[
  1\leq D\left(\frac{D}{2}\right)^{m-1}
   =\frac{D^m}{2^{m-1}}.
\]
Thus $D^m\geq 2^{m-1}$, and taking the positive $m$-th root proves (6).
\end{proof}

\begin{remark}
\label{rem:universality}
The proof uses neither the total mass nor any geometric structure. Only positivity, conservation under a binary split, and the number of refinement steps enter the argument. In particular, normalizing the initial mass to $1$ has no effect on the result.
\end{remark}

\section{Exact finite interval discrepancy}
\label{sec:interval}

Every interval strategy induces a binary refinement process by taking the pieces to be the intervals and their weights to be their lengths. Splitting an interval of length $x$ into intervals of lengths $y$ and $z$ satisfies $y+z=x$, so Theorem~\ref{thm:binary} applies immediately.

\begin{corollary}
\label{cor:interval-lower}
For every positive integer $n$,
\[
  \disc(n)\geq 2^{1-1/\ceil{n/2}}.
  \tag{9}
\]
\end{corollary}

\begin{proof}
Apply Theorem~\ref{thm:binary} to the sequence of interval lengths of an arbitrary strategy of length $n$, and then take the minimum over all strategies.
\end{proof}

The abstract mass model has a converse geometric realization.

\begin{lemma}
\label{lem:interval-realization}
After normalizing the initial mass to $1$, every binary refinement process of
length $n$ has a realization by intervals in $[0,1]$ whose lengths agree with
the corresponding masses at every stage. In particular, the discrepancies of
the two processes agree at every stage.
\end{lemma}

\begin{proof}
Assign the initial piece the interval $[0,1]$. Suppose that, at some stage,
the parent piece $x$ has been assigned $[a,b]$ and has mass
$w_i(x)=b-a$. If its two children have positive masses $u$ and $v$ with
$u+v=w_i(x)$, replace $[a,b]$ by the adjacent intervals
$[a,a+u]$ and $[a+u,b]$. All other intervals remain unchanged. The new
intervals are nonempty and have lengths $u$ and $v$, so induction preserves
the partition and the claimed correspondence of lengths and masses.
\end{proof}

Thus the binary-refinement model is exact for the finite interval problem,
rather than merely a relaxation of it.

DeLeo, Henderschedt, and Wells describe lex-merge as a merging process.
Reversing that sequence gives a binary refinement of the basket weights, and
normalizing all basket lengths by their common total does not change any
discrepancy. By Lemma~\ref{lem:interval-realization}, this normalized mass
process has an interval realization with the same discrepancy. The resulting
lex-merge construction has
\[
  \disc(\mathrm{LM}_n)=2^{1-1/\ceil{n/2}}
  \tag{10}
\]
by Theorem 3.5 of \cite{DHW}. Combining (9) and (10) proves the main result.

\begin{theorem}
\label{thm:main}
For every positive integer $n$,
\[
  \boxed{\displaystyle \disc(n)=2^{1-1/\ceil{n/2}}.}
  \tag{11}
\]
In particular, lex-merge is optimal for every finite horizon.
\end{theorem}

\begin{proof}
Corollary~\ref{cor:interval-lower} gives the lower bound in (11), while (10) gives the matching upper bound.
\end{proof}

Thus Conjecture 1.4 of \cite{DHW} holds for every positive integer $n$.

The main results are also accompanied by a machine-checked Isabelle/HOL
formalization \cite{formalization}. The development formalizes the
midpoint-survivor lower bound for labeled mass states, verifies that the full
lexicographic basket trace agrees with the numerical queue trace used in the
upper-bound proof, and establishes both the projection from interval
bisections to mass refinements and the realization of every abstract strategy
by genuine intervals. Thus it checks that the exact theorem concerns the
named lex-merge construction and the original interval problem, rather than
only an auxiliary numerical model. It is provided as a supplementary
verification artifact for the argument in this note.

\begin{corollary}
\label{cor:asymptotic}
As $n\to\infty$,
\[
  \disc(n)=2-\frac{4\log 2}{n}+O\!\left(\frac1{n^2}\right).
  \tag{12}
\]
\end{corollary}

\begin{proof}
Let $m=\ceil{n/2}$. Then $1/m=2/n+O(1/n^2)$, and hence
\[
  2^{1-1/m}
  =2\exp\!\left(-\frac{\log 2}{m}\right)
  =2-\frac{2\log 2}{m}+O\!\left(\frac1{m^2}\right)
  =2-\frac{4\log 2}{n}+O\!\left(\frac1{n^2}\right).
\]
\end{proof}

The lower bound in \cite{DHW} was obtained through a structural analysis of the ordered interval lengths and produced the scale $\ceil{n/3}$. The refinement argument above bypasses that ordering information entirely. The optimal $\ceil{n/2}$ scale arises instead from the existence of a piece that survives from just beyond the midpoint of the process to its terminal state.

\section{\texorpdfstring{An $r$-ary extension}{An r-ary extension}}
\label{sec:rary}

The same argument applies when every refinement produces more than two children. We record the extension because it isolates the combinatorial principle behind Theorem~\ref{thm:binary}.

Fix an integer $r\geq 2$. An $r$-ary refinement process begins with one positive piece and, at each step, replaces one piece by $r$ positive pieces whose weights sum to the weight of their parent. Suppose that the process terminates with $N$ pieces. Necessarily
\[
  N=1+(r-1)T
  \tag{13}
\]
for some number $T$ of refinement steps. Let $D$ again denote the largest, over all states, of the ratio between the maximum and minimum weights.

\begin{proposition}
\label{prop:rary}
Let $m=\ceil{N/r}$. Every $r$-ary refinement process terminating with $N$ pieces satisfies
\[
  D\geq r^{1-1/m}.
  \tag{14}
\]
\end{proposition}

\begin{proof}
Let $\alpha_t$ denote the minimum weight after $t$ refinements, for $0\leq t\leq T$. At one refinement, the parent has weight at most $D\alpha_t$, and at least one of its $r$ children has weight at most $1/r$ of the parent weight. Thus successive minima satisfy
\[
  \alpha_{t+1}\leq \frac{D}{r}\alpha_t.
  \tag{15}
\]

The process has at least $m-1$ refinements because $m-1\leq T$; for $T\geq 1$, this follows from
\[
  \frac Nr=T-\frac{T-1}{r}\leq T,
\]
and the case $T=0$ is immediate. Consider the state with exactly $m-1$ refinements remaining. It contains
\[
  N-(r-1)(m-1)
\]
pieces. Since $m-1<N/r$, this number is strictly greater than $m-1$, the number of refinements remaining. Therefore at least one current piece survives unchanged to the terminal state.

The state under consideration is after $T-(m-1)$ refinements. The survivor gives
\[
  \alpha_{T-(m-1)}\leq D\alpha_T,
\]
while iterating (15) through the remaining $m-1$ refinements gives
\[
  \alpha_T\leq \left(\frac{D}{r}\right)^{m-1}\alpha_{T-(m-1)}.
\]
Consequently,
\[
  1\leq D\left(\frac{D}{r}\right)^{m-1}
   =\frac{D^m}{r^{m-1}},
\]
which proves (14).
\end{proof}

For $r=2$, Proposition~\ref{prop:rary} is exactly Theorem~\ref{thm:binary}. We do not claim that the bound in (14) is sharp for every $r>2$.

\section{Concluding remarks}
\label{sec:conclusion}

The exact finite interval discrepancy is therefore
\[
  \disc(n)=2^{1-1/\ceil{n/2}}.
\]
The proof separates the problem into two independent components. The upper bound is constructive and is supplied by lex-merge. The lower bound is universal: it is forced by binary refinement itself and is independent of the geometry and ordering of the intervals.

The argument determines the optimal value but does not characterize all strategies attaining it. A natural further problem is to describe the equality cases in Theorem~\ref{thm:binary} and, in particular, to determine to what extent optimal interval strategies are forced to have the structure exhibited by lex-merge.


\begin{thebibliography}{9}
\small

\bibitem{deBruijnErdos}
N.~G. de Bruijn and P.~Erd\H{o}s.
\newblock Sequences of points on a circle.
\newblock \emph{Proceedings of the Section of Sciences of the Koninklijke Nederlandse Akademie van Wetenschappen te Amsterdam} 52 (1949), 14--17.

\bibitem{DHW}
Jared DeLeo, Owen Henderschedt, and Chris Wells.
\newblock A finite victory over de Bruijn--Erdos in interval discrepancy.
\newblock arXiv:2605.29166 [math.CO], 2026.
\newblock \url{https://doi.org/10.48550/arXiv.2605.29166}.

\bibitem{formalization}
Arthur F. Ramos.
\newblock Lex-Merge Is Optimal for Finite Interval Discrepancy: Isabelle/HOL formalization, version 1.0.1.
\newblock Zenodo, 2026.
\newblock \url{https://doi.org/10.5281/zenodo.21856255}.

\end{thebibliography}
\end{document}